\documentclass[12pt]{elsarticle}
\usepackage{amsmath, amsthm,
amssymb,mathdots,extsizes}
\usepackage[all]{xy}
\usepackage[active]{srcltx}

\DeclareMathOperator{\ind}{ind}
\DeclareMathOperator{\End}{End}
\DeclareMathOperator{\rad}{Rad}
\DeclareMathOperator{\diag}{diag}

\renewcommand{\ge}{\geqslant}

\newcommand{\cc}{\mathbb C}
\newcommand{\rr}{\mathbb R}
\newcommand{\ff}{\mathbb F}
\newcommand{\pp}{\mathbb P}

\newcommand{\mc}{\mathcal}

\newcommand{\arr}[1]{\left[\begin{array}
 #1  \end{array}\right]}

\newcommand{\matt}[1]{\left[\begin{smallmatrix}
 #1 \end{smallmatrix}\right]}
\newcommand{\mat}[1]{\begin{bmatrix}
 #1  \end{bmatrix}}
\newcommand{\ma}[1]{\begin{matrix}
 #1  \end{matrix}}

\newtheorem{theorem}{Theorem}
\newtheorem{lemma}{Lemma}
\theoremstyle{remark}
\newtheorem{remark}{Remark}

\begin{document}

\title{Symplectic spaces and pairs of symmetric and nonsingular skew-symmetric matrices under congruence\thanks{aa}}

\author[bov]{Victor A.~Bovdi} \ead{vbovdi@gmail.com}
\address[bov]{United Arab Emirates University, Al Ain, UAE}

\author[hor]{Roger A.~Horn} \ead{rhorn@math.utah.edu}
\address[hor]{Department of
Mathematics, University of Utah, Salt
Lake City, Utah 84103}

\author[bov]{Mohamed A.~Salim}
\ead{msalim@uaeu.ac.ae}

\author[ser]{Vladimir~V.~Sergeichuk\corref{cor}}
\ead{sergeich@imath.kiev.ua}
\address[ser]{Institute of Mathematics, Tereshchenkivska 3,
Kiev, Ukraine}

\cortext[cor]{Corresponding author. Published in: Linear Algebra Appl. 537 (2018) 84--99.}

\begin{abstract} Let $\ff$ be a field of characteristic not $2$, and let $(A,B)$ be a pair of $n\times n$ matrices over $\ff$, in which $A$ is  symmetric and $B$ is skew-symmetric. A canonical form of
$(A,B)$ with respect to congruence transformations $(S^TAS,S^TBS)$ was given by Sergeichuk (1988) up to classification of symmetric and Hermitian forms over finite extensions of $\ff$. We obtain a simpler canonical form of $(A,B)$ if $B$ is nonsingular. Such a pair $(A,B)$  defines a quadratic form on a symplectic space, that is, on a vector space with scalar product given by a nonsingular skew-symmetric form. As an application, we obtain known canonical matrices of quadratic forms and Hamiltonian operators on
real and complex symplectic spaces.
\end{abstract}


\maketitle

\section{Introduction}

In this paper, $\ff$ denotes a field of characteristic not $2$. We consider the canonical form problem for pairs of $n\times n$ matrices
\begin{equation}\label{111}
 (A,B)\text{ with symmetric $A$ and  nonsingular skew-symmetric $B$}
\end{equation}
over $\ff$
with respect to congruence transformations
$$
 (A,B)\mapsto (S^TAS,S^TBS),\qquad\text{$S\in\ff^{n\times n}$ is nonsingular.}$$

This is the problem of classifying quadratic forms $(u,u)\mapsto u^TAu$ on the space $\ff^n$ with scalar product $(u,v)\mapsto u^TBv$. A vector space with scalar product given by a nonsingular skew-symmetric form is called \emph{symplectic}. Thus, we consider the problem of classifying quadratic forms on symplectic spaces.

Canonical forms under congruence for matrix pairs \eqref{111} over $\rr$ and $\cc$ without the condition of nonsingularity of $B$ are given in \cite{goh,lan1,thom}.
A more general problem of classifying pairs of symmetric or skew-symmetric matrices with respect to congruence (or pairs of Hermitian matrices with respect to *congruence)
was solved by Sergeichuk \cite[Theorem 4]{ser_izv} over $\ff$ up to classification of symmetric and Hermitian forms over finite extensions of $\ff$ (without ``up to \dots'' if $\ff$ is $\rr$ or $\cc$). Closely related results for pairs of Hermitian matrices were obtained by Williamson \cite{will1,will2}.

Theorem 4 in \cite{ser_izv} is proved by a universal method that is valid for arbitrary systems of forms and linear mappings over a field or skew field of characteristic not 2; it reduces the problem of classifying such a system to the problem of classifying quadratic and Hermitian forms and some system of linear mappings (i.e., representations of some quiver). This method  is developed in \cite{roi,ser_izv} and is applied in  \cite{ser_izv} (see also \cite{hor-ser_bilin,hor-ser_mixed,mel,ser_brazil,ser_isom}) to the problem of classifying bilinear or sesquilinear forms, pairs of symmetric or skew-symmetric or Hermitian forms,  as well as to the problem of classifying isometric or selfadjoint  operators on a space with nonsingular symmetric or Hermitian form. We now apply this method to the problem of classifying quadratic forms and Hamiltonian operators on symplectic spaces.

For pairs $(A,B)$ with symmetric $A$ and skew-symmetric $B$, the canonical form in \cite[Theorem 4]{ser_izv} simplifies significantly if $A$ is nonsingular, but, unfortunately, not if $B$ is nonsingular.

In Section \ref{ss2} we modify the proof of Theorem 4 in \cite{ser_izv} in order to obtain a simple canonical  form of \eqref{111}.

John Williamson in \cite{will}  classified
matrix pairs \eqref{111} over $\rr$ up to congruence; he obtained some results over any field. As an application, he constructed for  pairs of $4\times 4$ real matrices \eqref{111}  their canonical pairs in the form $(A_{\text{can}},\Omega_n)$ with $n=2$ and
 \begin{equation*}\label{fel}
\Omega_n:=\mat{0&I_n\\-I_n&0}.
\end{equation*}
Thus, $A_{\text{can}}$ is a canonical form of a $4\times 4$ symmetric real matrix $A$ with respect to transformations
\[
A\mapsto S^TAS\qquad\text{such that $S^T\Omega_nS=\Omega_n$ and $S$ is nonsingular.}
\]
These transformations are called \emph{symplectic congruence transformations}; a matrix $S$ such that $S^T\Omega_nS=\Omega_n$ is called \emph{symplectic}. Thus, Williamson constructed canonical matrices of quadratic forms on a four-dimensional real symplectic space in a \emph{symplectic basis} (i.e., in a basis in which the scalar product is given by the matrix $\Omega_n$).

The following consequence of Williamson's classification of matrix pairs \eqref{111} over $\rr$ is known as
\emph{Williamson's theorem } \cite{arn,bha,gos}:
\begin{equation}\label{ere}
\parbox[c]{0.8\textwidth}{\emph{each $2n\times 2n$ positive definite symmetric real matrix $A$ is symplectically
congruent to a unique diagonal matrix
$ D\oplus D$, in which $D=\diag(\alpha_1,\dots,\alpha_n)$ and $\alpha_1\ge\dots\ge \alpha_n>0$.}
}
\end{equation}
Hofer and Zehnder \cite{hof} point out that this fact was known to Weierstrass \cite{wei}. The positive numbers $\alpha_1,\dots,\alpha_n$ are the \emph{symplectic eigenvalues} of $A$;  their properties are studied by Bhatia and Jain \cite{bha}.

Using Williamson's  classification \cite{will},  Galin \cite{gal} constructs canonical quadratic forms (in the form of homogeneous polynomials of degree two in $2n$ variables) on the real symplectic space $\rr^{2n}$ with scalar product given by $\Omega_n$. Galin's list was presented by Arnold  in \cite[Appendix 6]{arn} and by Arnold and  Givental in \cite[Chapter 1, \S\,2.4]{arn1}.

A linear operator $\mc H$ on a symplectic space $V$ with scalar product $\omega (u,v)$ is  \emph{Hamiltonian} if $\omega (\mc H u,v) =-\omega (u,\mc H v)$ for all $u,v\in V$.  It defines the form $\mc A(u,v):=  \omega (\mc H u,v)$ on $V$, which is symmetric since
 \begin{equation*}\label{eef}
\mc A(u,v)=  \omega (\mc H u,v)=-\omega (u,\mc H v)= \omega (\mc H v,u)=\mc A(v,u).
\end{equation*}
Therefore, the problem of classifying quadratic forms on symplectic spaces is equivalent to the problem of classifying Hamiltonian operators. If $H$ is the matrix of $\mc H$ in a symplectic basis, then $\mc A(u,v)=-\omega (u, \mc H  v)=-[u]^T\Omega_nH[v]$, hence $-\Omega_nH$ is the matrix of the symmetric form $\mc A(u,v)$.
Since $\Omega_n(-\Omega_nH)=H$,
\begin{equation}\label{xxx}
\parbox[c]{0.8\textwidth}{left multiplication by $\Omega_n$ of a set of canonical matrices of symmetric forms on symplectic spaces produces  a set of canonical matrices of Hamiltonian operators.}
\end{equation}

Thus, canonical matrices of real Hamiltonian operators can be obtained by left multiplication by $\Omega_n$ of Galin's canonical matrices of quadratic forms. Burgoyne and Cushman \cite{burg,burg1} construct
another set of canonical matrices of real Hamiltonian operators using Williamson's article \cite{will}. Canonical matrices of real Hamiltonian operators whose spectrum is an imaginary pair are constructed by Coleman \cite{col}.
Miniversal deformations of Burgoyne and Cushman's canonical matrices are obtained by Ko\c{c}ak \cite{koc}; see also \cite{gal} and \cite[Chapter 1, \S\,3]{arn1}.

A Hamiltonian operator is given in a symplectic basis by a \emph{Hamiltonian matrix}; that is, by a $2n\times 2n$ matrix $H$ such that $\Omega_nH$ is symmetric. The problem of classifying Hamiltonian operators is the problem of classifying Hamiltonian matrices with respect to \emph{symplectic similarity transformations}
\[
H\mapsto S^{-1}HS,\qquad S\text{ is symplectic.}
\]
Arnold and  Givental \cite[Chapter 1, \S\,2.2]{arn1}  deduce from  \cite{will} that
\emph{complex Hamiltonian matrices are  symplectically similar if and  only  if they  are  similar } (an analogous statement is proved in \cite[Problem P.12.4]{gar} and \cite[Problem 2.5.P16]{hor-joh}: complex Hermitian matrices are unitarily similar if and only if they are similar).
Pairs of Hamiltonian or symplectic matrices are studied in \cite{fer,lin,mehl}.

In Section \ref{ss2} we obtain
canonical forms of \eqref{111} up to classification of symmetric and Hermitian forms. In Section \ref{sen} we give them without ``up to'' over an algebraically or real closed field of characteristic not 2 (in particular, over $\cc$ and $\rr$).  We give them  in the form $(A_{\text{can}},\Omega_n)$.
Thus, $A_{\text{can}}$ is a canonical form of a symmetric matrix $A$ with respect to symplectic congruence. After left multiplication of all $A_{\text{can}}$ by $\Omega_n$ (see \eqref{xxx}), we obtain canonical matrices of Hamiltonian operators. They practically coincide with Burgoyne and Cushman's canonical matrices \cite{burg,burg1}, but we obtain them using Theorem \ref{thm1} instead of  Williamson's classification of quadratic forms on symplectic spaces \cite{will}.

\section{Matrix pairs over a field of characteristic not 2}\label{ss2}

A \emph{real closed field}  is a field $\pp$ that has index 2 in its algebraic closure $\overline\pp$. Its properties are close to the properties of $\rr$ (see \cite[Chapter VI, \S\,2]{gri} or \cite[Lemma 2.1]{hor-ser_bilin}):  the characteristic of $\pp$ is $0$ and each algebraically closed field of characteristic $0$ contains a real closed subfield; there is $i\in\overline\pp$ such that $i^2=-1$ and $\overline\pp=\pp(i)$; $\pp$ has a unique linear ordering such that
$a > 0$ and $b > 0$ imply $a + b > 0$ and $ab> 0$.

A \emph{Frobenius block} over a field is a matrix
\begin{equation*}\label{3.lfo}
\Phi=\mat{
0&&0&-c_n\\[-5pt]
1&\ddots&&\vdots\\[-8pt]&\ddots&0&-c_2
\\0&&1& -c_1
}
\end{equation*}
whose characteristic
polynomial
$\chi_{\Phi}(x)$ is a
power of an irreducible
polynomial
$p_{\Phi}(x)$:
\begin{equation*}\label{ser24}
\chi_{\Phi}(x)=p_{\Phi}(x)^s
=x^n+
c_1x^{n-1}+\dots+c_n.
\end{equation*}
Each square matrix over any field is similar to a direct sum of Frobenius blocks and this sum is uniquely determined up to permutations of summands.

We denote by $\mc F_n(\ff)$ any set of $n\times n$ matrices over $\ff$ obtained by replacement of each $n\times n$ Frobenius block by a similar matrix.
If $\ff$ is an algebraically closed field, then $\mc F_n(\ff)$  can be the set of  Jordan blocks
\begin{equation*}\label{def}
J_n(a):=\mat{a&1&&0\\[-8pt]&a&\ddots\\[-8pt]&&\ddots&1
\\0&&&a}\quad(\text{$n$-by-$n$,  } a\in\ff).
\end{equation*}
 If $\pp$ is a real closed field and $\overline\pp=\pp(i)$ with $i^2=-1$, then $\mc F_n(\pp)$ can be the set of all $J_n(a)$ with $a\in\pp$,  and of the realifications
\begin{equation}\label{tys}
J_n(a+bi)^{\pp}:=\mat{a&b&1&0&&&&0
\\-b&a&0&1
\\[-8pt]&&a&b&\ddots&
\\[-8pt]&&-b&a&&\ddots
\\[-8pt]&&&&\ddots&&1&0
\\[-8pt]&&&&&\ddots&0&1
\\&&&&&&a&b
\\0&&&&&&-b&a}
\quad
       (2n\text{-by-}2n)
\end{equation}
of Jordan blocks $J_n(a+bi)$ over $\overline\pp$,
in which $a,b\in\pp$ and $b>0$.

For each
$\Psi\in \mc F_n(\ff)$, if there
exists a nonsingular skew-symmetric matrix
$M$ such that $ M\Psi$ is symmetric, then we fix one and denote it by $\widetilde\Psi$. Thus,
\begin{equation}\label{nhg}
\widetilde\Psi\Psi=(\widetilde\Psi\Psi)^T,\qquad  \widetilde\Psi=-\widetilde\Psi^T.
\end{equation}
Denote by $\widetilde{\mc F}_{n}(\ff)$ the set of all $\Psi\in \mc F_n(\ff)$ for which $\widetilde \Psi$ exists.

\begin{theorem}\label{thm1}
Let $A$ be  a symmetric matrix  and let $B$ be a nonsingular skew-symmetric matrix of the same size over a field\/ $\ff$ of characteristic not $2$.
\begin{itemize}
  \item[\rm(a)] The pair
$(A,B)$ is congruent over\/ $\ff$ to a direct sum of pairs of the following two types:
\begin{itemize}
  \item[\rm(i)] $\mc P_{\Phi}:=\left(
  \mat{0&\Phi\\\Phi^T&0},\mat{0&I_n\\-I_n&0} \right)$, in which  $\Phi\in{\mc F}_{n}(\ff)\smallsetminus\widetilde{\mc F}_{n}(\ff)$. These summands are uniquely determined, up to replacement of some summands $\mc P_{\Phi}$ with $\mc P_{\Phi'}$, in which $\Phi'\in{\mc F}_{n}(\ff)$ is such that
$\chi_{\Phi'}(x)
=(-1)^n\chi_{\Phi}(- x)$.

  \item[\rm(ii)]  $\mc Q_{\Psi}^{f(x)}:=(\widetilde \Psi\Psi,\widetilde \Psi)f(\Psi)=(\widetilde \Psi\Psi f(\Psi),\widetilde \Psi f(\Psi))$, in which $\Psi\in\widetilde{\mc F}_{n}(\ff)$, $0\ne f(x)\in\ff[x^2]$, and $\deg f(x)<\deg p_{\Psi}(x)$. These summands are uniquely determined, up to replacement of the
whole group of summands
\[
\mc Q_{\Psi}^{f_1(x)}
\oplus\dots\oplus
\mc Q_{\Psi}^{f_s(x)}
\]
with the same $\Psi$ by
\[
\mc Q_{\Psi}^{g_1(x)}
\oplus\dots\oplus
 \mc Q_{\Psi}^{g_s(x)},
\]
in which
the Hermitian form
\[
g_1(\omega)x_1^{\circ}x_1+\dots+
g_s(\omega)x_s^{\circ}x_s
\]
is equivalent to the Hermitian form
\begin{equation}\label{777aki}
f_1(\omega)x_1^{\circ}x_1+\dots+
f_s(\omega)x_s^{\circ}x_s
\end{equation}
over the
field\/ ${ \ff}(\omega):={ \ff}[x]/p_{\Psi}(x){ \ff}[x]$
with involution
$f(\omega)\mapsto f(\omega)^{\circ}:=
f(-\omega)$, in which $\omega:=x+p_{\Psi}(x){ \ff}[x]$.
\end{itemize}

  \item[\rm(b)] If\/ $\ff$ is an algebraically closed field of characteristic not $2$, then {\rm(ii)} can be replaced by
\begin{itemize}
  \item[\rm(ii$'$)] $(\widetilde \Psi\Psi,\widetilde \Psi)$, in which $ \Psi\in\widetilde{\mc F}_{n}(\ff)$.
\end{itemize}
If\/ $\ff$ is a real closed field, then {\rm(ii)} can be replaced by
\begin{itemize}
  \item[\rm(ii$''$)] $(\widetilde \Psi\Psi,\widetilde \Psi)$, $(-\widetilde \Psi\Psi,-\widetilde \Psi)$, in which $ \Psi\in\widetilde{\mc F}_{n}(\ff)$.
\end{itemize}
The summands {\rm(ii$'$)} and {\rm(ii$''$)} are determined uniquely, up to permutations.
\end{itemize}
\end{theorem}

The following lemma is analogous to \cite[Theorem 8]{ser_izv}.

\begin{lemma}\label{erd}
\begin{itemize}
  \item [\rm(a)]
Let\/ $\ff$ be a field of characteristic not $2$ and let $\Psi\in{\mc F}_{n}(\ff)$. The matrix $\widetilde\Psi$ exists if and only if
\begin{equation}\label{ff}
\text{$n$ is even and either $p_{\Psi}(x)=x$ or $p_{\Psi}(x)\in\ff[x^2]$.}
\end{equation}

  \item[\rm(b)]
If\/ $\ff$ is an algebraically closed field of characteristic not $2$, then \eqref{ff} can be replaced by
\begin{equation}\label{ff1}
\text{$n$ is even and $p_{\Psi}(x)=x$.}
\end{equation}
If\/ $\ff$ is a real closed field, then \eqref{ff} can be replaced by
\begin{equation}\label{ff2}
\text{$n$ is even and either $p_{\Psi}(x)=x$ or $p_{\Psi}(x)=x^2+b$ with $b>0$.}
\end{equation}
\end{itemize}
\end{lemma}

\begin{proof}[Constructive proof] (a) Suppose that $\widetilde\Psi$ exists.  Since the $n\times n$ matrix $\widetilde\Psi$ is nonsingular and skew symmetric, $n$ is even. Since $\widetilde\Psi\Psi=(\widetilde\Psi\Psi)^T =\Psi^T\widetilde\Psi^T=-\Psi^T\widetilde\Psi,
$ $\Psi$ is similar to $-\Psi^T$. Hence $\det(xI_n-\Psi)= \det(xI_n+\Psi)=(-1)^n\det(-xI_n-\Psi)$, $\chi_{\Psi}(x)=(-1)^n\chi_{\Psi}(-x)$, and we have $p_{\Psi}(x)=\pm p_{\Psi}(-x)$. If $\deg p_{\Psi}(x)$ is odd, then $p_{\Psi}(x)=- p_{\Psi}(-x)$, $p_{\Psi}(0)=- p_{\Psi}(0)=0$, $x|p_{\Psi}(x)$, and so $p_{\Psi}(x)=x$.  If $\deg p_{\Psi}(x)$ is even, then $p_{\Psi}(x)=p_{\Psi}(-x)$, and so $p_{\Psi}(x)\in\ff[x^2]$. We have proved  \eqref{ff}.

Conversely, suppose that  \eqref{ff} holds.

If  $p_{\Psi}(x)=x$, then $\Psi=S^{-1}J_n(0)S$ with a nonsingular $S$. We take
\[
\widetilde\Psi:=S^T\matt{0&&&&1\\&&&-1\\&&1\\&-1
\\ \iddots&&&&0}S.
\]
Then
\[
\widetilde\Psi\Psi=S^T\matt{0&&&&0\\&&&0&-1
\\&&0&1\\&0&-1
\\ \iddots&\iddots&&&0}S,
\]
and we have proved \eqref{nhg}.

Now let $p_{\Psi}(x)\in\ff[x^2]$. Let $\Phi$ be the Frobenius block that is similar to $\Psi^{-1}$; that is, $\Psi=S^{-1}\Phi^{-1} S$ for a nonsingular $S$. Then
\begin{align*}
\chi_{\Phi}(x)&=\det(xI_n-\Psi^{-1})=\det(\Psi^{-1})
x^n\det(\Psi-x^{-1}I_n)\\
&=\det(\Psi^{-1})
x^n\chi_{\Psi}(x^{-1})
\in\ff[x^2].
\end{align*}
The existence of a nonsingular matrix $\Gamma$ such that
\[
\Gamma=\Gamma^T,\qquad
\Gamma\Phi=-(\Gamma\Phi)^T
\]
is proved and $\Gamma$ is constructed explicitly in \cite[Theorem 8]{ser_izv}.
The equalities \eqref{nhg} hold for $\widetilde\Psi:=S^T\Gamma\Phi S$ since
\begin{align*}
\widetilde\Psi^T&=S^T(\Gamma\Phi)^T S = -S^T(\Gamma\Phi) S=-\widetilde\Psi^T,
\\
(\widetilde\Psi\Psi)^T&=
(S^T\Gamma\Phi S\cdot
S^{-1}\Phi^{-1} S)^T=
(S^T\Gamma S)^T=S^T\Gamma^TS
         \\&=S^T\Gamma S=
 S^T\Gamma\Phi S\cdot
S^{-1}\Phi^{-1} S=\widetilde\Psi\Psi.
\end{align*}

(b) This statement is obvious.
\end{proof}

\begin{proof}[Proof of Theorem \ref{thm1}]
By the proof of \cite[Theorem 4]{ser_izv},
\begin{equation}\label{gtl}
\parbox[c]{0.8\textwidth}{the problem of classifying pairs of bilinear forms $\mc A,\mc B:U\times U\to\ff$ over $\ff$, in which $\mc A$ is  symmetric and $\mc B$ is skew-symmetric, reduces to the problem of classifying pairs of linear mappings
$\mc M,\mc N :  V\to W$  over $\ff$. If $\mc B$ is nonsingular, then $\mc N$ is bijective. }
\end{equation}
By Kronecker's theorem for matrix pencils, for each pair of linear mappings $\mc M,\mc N:V\to   W$ there exist bases of $V$ and $W$, in which the pair $(M,N)$ of matrices of $\mc M$ and $\mc N$ is a direct sum of pairs of the form
\begin{equation}\label{smy}
(I_n,\Phi),\ \ (J_n(0),I_n),\ \ (L_n,R_n),\ \ (L_n^T,R_n^T),\qquad n=1,2,\dots,
\end{equation}
where $\Phi\in \mc F_n(\ff)$ and
\[
L_n:=\mat{1&0&&0\\[-8pt]&\ddots&\ddots\\0&&1&0},\quad
R_n:=\mat{0&1&&0\\[-8pt]&\ddots&\ddots\\0&&0&1}
\qquad((n-1)\text{-by-}n),
\]
and this sum is uniquely determined by $\mc M$ and $\mc N$, up to permutations of summands.

(a) The set of pairs \eqref{smy} is used in the proof of \cite[Theorem 4]{ser_izv}.
Let $\mc B$ in \eqref{gtl} be nonsingular; then $\mc N$ is bijective.  By \eqref{smy}, the pair $(M,N)$ of matrices of $\mc M$ and $\mc N$ is a direct sum of pairs of two types:
\begin{equation}\label{lki}
\text{$(I_n,\Phi)$ with nonsingular $\Phi$,\qquad $ (J_n(0),I_n)$.}
\end{equation}
They give four types of summands.

In contrast, we use summands of the type $(\Phi,I_n)$ with any $\Phi\in \mc F_n(\ff)$  instead of \eqref{lki} and obtain only the two types of summands in Theorem \ref{thm1}. The reasoning for $(\Phi,I_n)$ is  analogous to the  reasoning for $(I_n,\Phi)$  in the proof of \cite[Theorem 4]{ser_izv}:  we consistently find (in the definitions and notations from \cite{ser_izv}) that
\begin{itemize}
  \item
$
\overline{S}:\xymatrix{
 {v}
 \ar@/^1.5pc/@{->}[rr]^{\alpha}
  \ar@/^/@{->}[rr]^{\alpha^*}
 \ar@/_/@{->}[rr]_{\beta}
 \ar@/_1.5pc/@{->}[rr]_{\beta^*} &&{v^*}
 }$ in which $
 \alpha=\alpha^*$,
 $\beta=-\beta^*$, the arrows $\beta$ and $\beta^*$ are assigned by bijective mappings, and representations of the quiver $ \overline{S}$ can be given by pairs of matrices of the same size;

  \item $\ind(\overline S)$ consists of the pairs $(\Phi,I_n)$, in which $\Phi\in \mc F_n(\ff)$;

  \item $\ind_1(\overline S)$ consists of the pairs $(\Phi,I_n)$, in which $\Phi\in{\mc F}_{n}(\ff)\smallsetminus\widetilde{\mc F}_{n}(\ff)$ and $\Phi$ is determined up to replacement
by $\Phi'\in{\mc F}_{n}(\ff)$ with
$\chi_{\Phi'}(x)
=(-1)^n\chi_{\Phi}(- x)$;

  \item $\ind_0(\overline S)$ consists of the pairs $\mc Q_{\Psi}:=(\widetilde \Psi\Psi,\widetilde \Psi)$, in which $\Psi\in\widetilde{\mc F}_{n}(\ff)$;

 \item the pairs $g_f:=(f(\Psi),f(-\Psi^T))$ with $f(x)\in\ff[x]$ form the ring $\End(\mc Q_{\Psi})$ with involution $g_f\mapsto g_f^{\circ}=(f(-\Psi),f(\Psi^T))$;

\item  the field $\mathbb T(\mc Q_{\Psi}):=\End (\mc Q_{\Psi})/\rad(\End (\mc Q_{\Psi}))$ can be identified with the field $\ff(\omega ):=\ff[x]/p_{\Psi}(x)\ff[x]$ with involution $f(\omega)\mapsto f(\omega)^{\circ}:=f(-\omega)$;

 \item  the orbit of $\mc Q_{\Psi}$ consists of the pairs $\mc Q_{\Psi}^{f(\omega)}:=\mc Q_{\Psi}f(\Psi)$, in which $0\ne f(x)\in\ff[x^2]$ and $\deg f(x)<\deg p_{\Psi}(x)$,
\end{itemize}
and apply \cite[Theorem 1]{ser_izv}. \medskip

(b) We could use \cite[Theorem 2]{ser_izv}, but we give a direct proof.

Let $\ff$ be an algebraically closed field of characteristic not $2$, and let  $\Psi\in\widetilde{\mc F}_{n}(\ff)$. By \eqref{ff1}, $p_{\Psi}(x)=x$, and so  the field $\ff(\omega)= {\ff}[x]/p_{\Psi}(x){ \ff}[x]$   from Theorem \ref{thm1}(a) is $\ff$ with the identity involution. By \cite[Chapter XV, Theorem 3.1]{lan},  the symmetric form \eqref{777aki}  is equivalent to the form
$x_1^{\circ}x_1+\dots+
x_s^{\circ}x_s.$
Therefore, all summands $\mc Q_{\Psi}^{f (x)}=(\widetilde \Psi \Psi ,\widetilde \Psi )f(\Psi )$  in Theorem \ref{thm1}(a) are congruent to $\mc Q_{\Psi }^{1}=(\widetilde \Psi \Psi ,\widetilde \Psi )$.

Let $\ff$ be a real closed field, and let $\Psi\in\widetilde{\mc F}_{n}(\ff)$. By \eqref{ff2},  $\ff(\omega)= \ff$ if $p_{\Psi}(x)=x$ and $\ff(\omega)=\overline{\ff}$ if $p_{\Psi}(x)=x^2+b$ with $b>0$. In the second case, the involution $f(\omega)^{\circ}=
f(-\omega)$ is not the identity. By the law of inertia (see \cite[Chapter XV, Corollaries 4.3 and 5.3]{lan}),  the symmetric or Hermitian form \eqref{777aki}  is equivalent to exactly one form
\[
-x_1^{\circ}x_1-\dots-x_r^{\circ}x_r+
x_{r+1}^{\circ}x_{r+1}+\dots+
x_s^{\circ}x_s.
\]
Therefore, all summands $\mc Q_{\Psi }^{f (x)}=(\widetilde \Psi \Psi ,\widetilde \Psi )f(\Psi )$ in Theorem \ref{thm1}(a) are congruent to $\mc Q_{\Psi }^{\pm 1}=\pm(\widetilde \Psi \Psi ,\widetilde \Psi )$; these summands
are uniquely determined up to permutations.
\end{proof}

\section{Quadratic forms on symplectic spaces}\label{sen}

We consider $2n\times 2n$ matrices that are partitioned into four $n\times n$ blocks. The \emph{block-direct sum} of such matrices is defined as follows:
\[
\arr{{c|c}A_1&A_2\\\hline A_3&A_4}
\boxplus\dots\boxplus
\arr{{c|c}D_1&D_2\\\hline D_3&D_4}
:=
\arr{{c|c}A_1\oplus\dots\oplus D_1
&A_2\oplus\dots\oplus D_2
\\\hline A_3\oplus\dots\oplus D_3
&A_4\oplus\dots\oplus D_4
}.
\]

Define the  $2n\times 2n$ symmetric  matrix
\begin{equation*}\label{m1}
P_{n}:=\arr{{c|c}\ma{&&&1\\[-5pt]
&&\iddots\\&1\\[-5pt]1}&0\\
     \hline
     0&\ma{&&&0\\[-9pt]&&\iddots&1
\\[-7pt]&0&\iddots\\0&1}}
\end{equation*}
over any field. For each $c>0$ from
a real closed field $\pp$, define the  $2n\times 2n$  symmetric matrix $Q_{n}(c)$ over $\pp$:
\[
Q_{n}(c):=\arr{{cccccc|cccccc}
\!\!&&&&& c& 0&1&&&
   \\
&&&&-c&& & 0&1&&&
   \\
&&&c&&&& &0&1&&
  \\[-5pt]
&&-c&&&&& &&0&\ddots&
  \\[-5pt]
&\iddots&&&&&& &&&\ddots&1
  \\
c&&&&&&&   &&&&0
                 \\
                 \hline
\!\!0&&&&& & &&&&&c
    \\
\!\!1&0&&&&&&&&&-c
   \\
&1&0&&&&&&&c
    \\
&&1&0&&&&&-c
   \\[-5pt]
&&&\ddots&\ddots&&&\iddots
   \\[-2pt]
&&&&1&0&c
}\text{if $n$ is odd}
\]
and
 \begin{equation*}\label{m2}
Q_{n}(c):=\arr{{cccc|cccc}&&& & C&I_2&&
  \\[-9pt]
&&&& &C&\ddots&
   \\[-9pt]
&&&& &&\ddots&I_2
   \\
&&&& &&&C
   \\   \hline
\!\!\vphantom{A^{A^{A^a}}}C^T&&&&I_2&&&
   \\
\!\!I_2&C^T&&&&
    \\[-5pt]
&\ddots&\ddots&&&
    \\
&&I_2&C^T
} \text{if $n$ is even},
 \end{equation*}
in which
$C:=\matt{0&1\\-c^2&0}$ and
the unspecified entries are zero.

In the following theorem, we give canonical matrices of quadratic forms and Hamiltonian operators on a symplectic space over an algebraically closed or real closed field of characteristic not 2.  Analogous canonical matrices of quadratic forms and Hamiltonian operators on real symplectic spaces are given in \cite[Chapter 1, \S\,2.4]{arn1} and \cite{burg,burg1}; they are based on Williamson's article \cite{will}.
\begin{theorem}\label{thm}
\begin{itemize}
  \item[\rm(a)]
For each quadratic form on a symplectic space over an algebraically closed field\/ $\ff$  of characteristic not $2$ $($in particular, over the field of complex numbers$)$, there exists a symplectic basis in which its matrix is a block-direct sum, uniquely determined up to permutations of block-direct summands, of matrices of two types:
\begin{equation}\label{rrd}
 \arr{{c|c}0&J_n(a)\\
\hline\vphantom{A^{A^A}}J_n(a)^T&0}\ (a\in\ff),\quad P_{n},
\end{equation}
in which $a\ne 0$ if $n$ is even and $a$ is determined up to replacement by $-a$.

  \item[\rm(b)]
For each quadratic form on a symplectic space over a real closed field\/ $\pp$ $($in particular, over the field of real numbers$)$, there exists a symplectic basis in which its matrix is a block-direct sum, uniquely determined up to permutations of block-direct summands, of matrices of three types:
\begin{equation}\label{thu}
 \arr{{c|c}0&\Phi_n\\\hline
\vphantom{A^{A^{A^a}}}\Phi_n^T&0},\quad
\pm P_{n},\quad
\pm Q_{n}(c)\ (c\in\pp,\ c>0),
\end{equation}
in which
\begin{equation}\label{mho}
\Phi_n:=\left\{
         \begin{array}{l}
           J_n(a) \hbox{ with $a\in\pp$ and $a\ge 0$; $n$ is odd if $a=0$; or} \\
           J_{n/2}(a+bi)^{\pp} \hbox{ with $a,b\in\pp$, $a\ge 0$, $b>0$, and even $n$,}
         \end{array}
       \right.
\end{equation}
in which\/ $J_{n/2}(a+bi)^{\pp}$ is defined in  \eqref{tys}.

  \item[\rm(c)] For each Hamiltonian operator on a symplectic space over a real closed field\/ $\pp$, there is a symplectic basis in which its matrix is a block-direct sum, uniquely determined up to permutations of block-direct summands, of matrices of three types:
\begin{equation}\label{thh}
 \arr{{c|c}\Phi_n&0\\\hline
\vphantom{A^{A^{A^a}}}0&-\Phi_n^T},\quad
\pm \Omega_nP_{n},\quad
\pm \Omega_nQ_{n}(c) \ (c\in\pp,\ c>0),
\end{equation}
in which $\Phi_n$ is defined in \eqref{mho}.
\end{itemize}
\end{theorem}

Theorem \ref{thm} implies Williamson's theorem  \eqref{ere}. Indeed, if $A$ is a positive definite symmetric real matrix, then it is real symplectically congruent to a block-direct sum of positive definite matrices of types \eqref{thu}. Among the matrices \eqref{thu}, only $Q_1(c)=\diag(c,c)$ with $c>0$ is positive definite since the diagonal entries of each positive definite matrix are positive numbers. Hence, $A$ is symplectically congruent to $ D\oplus D$, in which $D=\diag(\alpha_1,\dots,\alpha_n)$ and $\alpha_1\ge\dots\ge \alpha_n>0$.

Our proof of Theorem \ref{thm} is based on the following lemma.

\begin{lemma}\label{ngi}\begin{itemize}
                          \item[\rm(a)]
If $n$ is odd, then $\widetilde{\mc F}_n(\ff)=\varnothing$. If\/ $\ff$ is an algebraically closed field of characteristic not $2$, then $\widetilde{\mc F}_{2n}(\ff)$ consists of one matrix;
we can take $\widetilde{\mc F}_{2n}(\ff)=\{-\Omega_nP_{n}\}$.
If\/ $\ff$ is a real closed field,  then we can take
$\widetilde{\mc F}_{2n}(\ff)=\{-\Omega_nP_{n},
-\Omega_nQ_{n}(c)|\,c>0\}$.

                          \item[\rm(b)]
If\/ $\widetilde{\mc F}_{2n}(\ff)$ is taken as in {\rm(a)} and $\Psi\in\widetilde{\mc F}_{2n}(\ff)$, then
 we can take $\widetilde \Psi:=\Omega_n$.
                        \end{itemize}
\end{lemma}

\begin{proof}
(a)
By \eqref{ff}, $\widetilde{\mc F}_n(\ff)$ is empty if $n$ is odd.

Let $\ff$ be an algebraically closed field of characteristic not $2$.  By \eqref{ff1},  $\widetilde{\mc F}_{2n}(\pp)$ consists of any matrix that is similar to $J_{2n}(0)$. The matrix $-\Omega_nP_n$ is nilpotent since $-(-\Omega_nP_n)^2=J_n(0)^T\oplus J_n(0)$. The rank of $-\Omega_nP_n$ is $2n-1$. Hence
$-\Omega_nP_n$ is similar to $J_{2n}(0)$.

Let $\ff=\pp$ be a real closed field, and let $\overline\pp=\pp(i)$ with $i^2=-1$ be its algebraic closure. Let us prove that we can take
$\widetilde{\mc F}_{2n}(\pp)=\{-\Omega_nP_{n},
-\Omega_nQ_{n}(c)|\,c>0\}$.
Due to \eqref{ff2}, it suffices to show that $-\Omega_nQ_{n}(c)$ is similar over $\overline\pp$ to $J_n(ci)\oplus J_n(-ci)$. Indeed, it suffices to show that $\Omega_nQ_{n}(c)$ is similar to $J_n(ci)\oplus J_n(-ci)$.

\medskip

$\bullet$
Let $n$ be even. We have
\[
\Phi:=S^{-1}\Omega_nQ_n(c) S=\arr{{cccc|cccc}
\!\! C&I&&&&&&
   \\[-8pt]
\!\!&C&\ddots&&&
    \\[-5pt]
&&\ddots&I&&
    \\
&&&C&F
   \\   \hline
\vphantom{A^{A^A}}
&&& & C&I&&\\[-9pt]
&&&& &C&\ddots&
   \\[-9pt]
&&&& &&\ddots&I
   \\
&&&& &&&C
},\quad F:=\mat{0&-1\\1&0},
\]
in which
\[
S:=\matt{&&1\\[-2pt]
&\iddots\\1}\oplus\diag(1,-1,-1,1,1,-1,-1,1,\ldots)
\]
is the direct sum
of two $n\times n$ matrices.

Write  \[R:=\mat{1&1\\ci&-ci}.\]
Then
\[R^{-1}=\frac1{2ci}\mat{ci&1\\ci&-1},
\qquad
R^{-1}CR=\mat{ci&0\\0&-ci},\]
\[G:=R^{-1}FR=\mat{1+c^2&1-c^2\\-1+c^2&-1-c^2},
\]
\begin{align*}
T:=&(R\oplus\dots\oplus R)^{-1}(\Phi-ciI_{2n})
(R\oplus\dots\oplus R)
\\=&
\arr{{cccc|cccc}
\!\! D&I&&&&&&
   \\[-8pt]
\!\!&D&\ddots&&&
    \\[-5pt]
&&\ddots&I&&
    \\
&&&D&G
   \\   \hline
\vphantom{A^{A^A}}
&&& & D&I&&\\[-9pt]
&&&& &D&\ddots&
   \\[-9pt]
&&&& &&\ddots&I
   \\
&&&& &&&D
}\text{ with } D:=\mat{0&0\\0&-2ci}.
\end{align*}
The matrix obtained from $T$ by deleting its first column and its penultimate row is nonsingular since
it has a block-triangular form in which all diagonal $2\times 2$ blocks are nonsingular. Hence  $\operatorname{rank} T=2n-1$, and so  $\Phi$ is similar to $J_n(ci)\oplus J_n(-ci)$ over  $\overline\pp$.
\medskip

$\bullet$
Let us show that $\Omega_nQ_n(c)$ with odd $n$ is similar to $J_n(ci)\oplus J_n(-ci)$. Expanding the determinant $\chi_n(x):=\det(xI_{2n}-\Omega_nQ_n(c)) $ along the first row, we get
 $\chi_n(x)=(x^2+c^2)\chi_{n-1}(x)$. Hence $\chi_n(x)=(x^2+c^2)^n$. In order to show that the rank of $
ciI_{2n}-\Omega_nQ_n(c)$ is $2n-1$, we delete its $(n+1)$st column and its last row and let $R_n$ denote the matrix  obtained. For example,
\[
R_5=\arr{{ccccc|cccc}
ci&0&0&0&0&0&0&0&-c\\
-1&ci&0&0&0&0&0&c&0\\
0&-1&ci&0&0&0&-c&0&0\\
0&0&-1&ci&0&c&0&0&0\\
0&0&0&-1&ci&0&0&0&0\\ \hline
0&0&0&0&c&1&0&0&0\\
0&0&0&-c&0&ci&1&0&0\\
0&0&-c&0&0&0&ci&1&0\\
0&-c&0&0&0&0&0&ci&1}
\]
 Let us prove that $R_n$
is nonsingular. Denote by $R_{n-1}$ the matrix obtained from $R_n$ by deleting its first and last rows and its first and last  columns. Expanding the determinant of $R_n$ along the first row, we find that  $\det (R_n)=2ci\det (R_{n-1})$ if $n> 3$. Hence $\det (R_n)=(2ci)^{n-2}\det (R_2)=(2ci)^{n-1}\ne 0$.
\medskip

(b)
We have proved that $\Omega_nQ_n(c)$ is similar to $J_n(ci)\oplus J_n(-ci)$.
For each $\Psi\in\{-\Omega_nP_{n},
-\Omega_nQ_{n}(c)|\,c>0\}$, the matrix $\Omega_n\Psi\in\{P_{n},
Q_{n}(c)\}$ is symmetric. Hence \eqref{nhg} holds with $\widetilde{\Psi}=\Omega_n$.
\end{proof}

\begin{proof}[Proof of Theorem \ref{thm}] Statement (a) follows from
the following  proof of (b).

(b) Let $\pp$ be a real closed field. Let us construct ${\mc F}_{2n}(\pp)$ as follows:  $\widetilde{\mc F}_{2n}(\pp)=\{P_{n},Q_{n}(c)|\,c>0\}$  as in Lemma \ref{ngi}, and  ${\mc F}_{2n}(\pp) \smallsetminus\widetilde{\mc F}_{2n}(\pp)$ consists of those Jordan blocks $J_{2n}(a)$ with $a\in\pp$ and the realifications \eqref{tys} that do not satisfy \eqref{ff2}.

Let $A\in\pp^{2m\times 2m}$ be any symmetric matrix of even size. By Theorem \ref{thm1} and Lemma \ref{ngi}, $(A,\Omega_m)$ is congruent to a direct sum of pairs of the form
\[
\left( \matt{0&\Phi_n\\
\vphantom{A^{A^A}}\Phi_n^T&0},
\Omega_n\right),\quad
\pm(P_{n},\Omega_n),\quad
\pm(Q_{n}(c),\Omega_n),
\]
in which $c>0$ and $\Phi_n$ is defined in \eqref{mho}.

Each pair $-(B,\Omega_n)=(-B,-\Omega_n)$ with symmetric $B\in\pp^{2n\times 2n}$ is congruent to $(-B,\Omega_n)$. This follows from the congruence of $M:=-B-\Omega_{n}$ and $M^T=-B^T-\Omega_{n}^T=-B+\Omega_{n}$  since these sums are the unique expressions of $M$ and $M^T$ as sums of symmetric and skew-symmetric matrices; the matrices  $M$ and $M^T$ are congruent since each square matrix over $\pp$ is congruent to its transpose (see \cite{hor-ser_tran}).

Thus, $(A,\Omega_m)$ is congruent to a direct sum of pairs of the form $(B,\Omega_n)$ with $B$ from \eqref{thu}; the summands
are determined by $(A,\Omega_{m})$ uniquely up to permutations. For each $M\in\rr^{2p\times 2p}$ and $N\in\rr^{2q\times 2q}$, the matrix pair
$(M, \Omega_{p})\oplus (N, \Omega_{q})$
is permutation congruent to
$(M\boxplus N, \Omega_p\boxplus \Omega_q)=(M\boxplus N, \Omega_{p+q})$.
Hence $(A,\Omega_m)$ is congruent to some pair $(D,\Omega_m)$, in which $D$ is a block-direct sum of matrices of the form \eqref{thu}.
\medskip

(c) If we replace each $\Phi_n\in{\mc F}_{n}(\pp) \smallsetminus\widetilde{\mc F}_{n}(\pp)$ in  part (b) of this proof by $\Phi_n^T$, we find that the set of matrices \eqref{thu} in Theorem \ref{thm}
can be replaced by
\begin{equation}\label{thu1}
 \arr{{c|c}0&\Phi_n^T\\\hline
\Phi_n&0},\quad
\pm P_{n},\quad
\pm Q_{n}(c).
\end{equation}
If we multiply the
matrices \eqref{thu1}
on the left by $\Omega_n$, we obtain the matrices \eqref{thh}, which proves (c) due to \eqref{xxx}.
\end{proof}

\begin{remark} In \eqref{rrd} and \eqref{thu}, the canonical summand $P_n$ can be replaced by
\[
P'_n:=\arr{{cccc|cccc}&&& & 0&1&&
  \\[-9pt]
&&&& &0&\ddots&
   \\[-9pt]
&&&& &&\ddots&1
   \\
&&&& &&&0
   \\   \hline
\!\!0&&&&1&&&
   \\
\!\!1&0&&&&
    \\[-5pt]
&\ddots&\ddots&&&
    \\
&&1&0}\quad\text{(as in \cite{koc})},
\]
and the canonical summand $Q_n(c)$ with odd $n$ can be replaced  by
\begin{equation*}\label{m3}
Q'_{n}(c):=\arr{{ccccc|ccccc}
\!\!\matt{c^2}&&&& & \matt{0}&\matt{0&1}&&
   \\
&&&&& & C&I_2&&
   \\[-9pt]
&&&&&& &C&\ddots&
  \\[-9pt]
&&&&&& &&\ddots&I_2
  \\
&&&&&& &&&C
                 \\
                 \hline
\!\!\matt{0}&&&& & \matt{1}&\matt{1&0}&&
    \\
\!\!\matt{0\\1}&C^T&&&&\matt{1\\0}&&&
   \\
&I_2&C^T&&&&
    \\[-5pt]
&&\ddots&\ddots&&&
   \\
&&&I_2&C^T
}
\end{equation*}
with $C:=\matt{0&1\\-c^2&0}$.
In order to prove this, let us show that $P_{n}$ and $Q_{n}(c)$ in Lemma \ref{ngi} can be replaced by $P'_{n}$ and $Q'_{n}(c)$, in which $Q'_{n}(c):=Q_{n}(c)$ if $n$ is even.

The matrix $\Omega_nP_n'$ is similar to $J_{2n}(0)$ since it acts on the standard basis, up to multiplication by $-1$, as follows:
\[
e_{2n}\mapsto e_{2n-1}\mapsto\cdots\mapsto e_{n+1}\mapsto
e_{1}\mapsto e_{2}\mapsto\cdots\mapsto e_{n}\mapsto 0.
\]
The matrix $\Omega_nQ'_n(c)$ with odd $n$ is similar to $J_n(ci)^{\pp}$ since
\[
S^{-1}\Omega_nQ'_n(c) S=\mat{
C&I
   \\[-8pt]
&C&\ddots
    \\[-5pt]
&&\ddots&I
    \\
&&&C}=J_n(ci)^{\pp},
\]
in which
\[
S:=\matt{&&1\\[-2pt]
&\iddots\\1}\oplus\diag(1,1,-1,-1,1,1,-1,-1,\ldots)
\]
is the direct sum of two $n\times n$ matrices.
\end{remark}

\section*{Acknowledgements}
The work was supported in part by the UAEU UPAR grants G00001922 and G00002160.

\end{document}